\documentclass[a4paper,12pt]{article}

\usepackage[utf8x]{inputenc}
\usepackage{amsfonts} 
\usepackage{amsmath} 
\usepackage[pdftex]{graphicx} 
\usepackage{caption} 
\usepackage{subcaption} 
\usepackage{sidecap} 
\usepackage{enumerate}
\usepackage{setspace}
\usepackage{MnSymbol}
\usepackage[toc,page]{appendix}
\usepackage{color}

\title{}
\author{}
\date{}

\newtheorem{theorem}{Theorem}[section]
\newtheorem{lemma}[theorem]{Lemma}
\newtheorem{proposition}[theorem]{Proposition}

\newenvironment{proof}[1][Proof]{\begin{trivlist}
\item[\hskip \labelsep {\bfseries #1}]}{\end{trivlist}}
\newenvironment{definition}[1][Definition]{\begin{trivlist}
\item[\hskip \labelsep {\bfseries #1}]}{\end{trivlist}}

\newenvironment{keywords}{\begin{@abssec}{\keywordsname}}{\end{@abssec}}
\newenvironment{@abssec}[1]{%
\if@twocolumn
\section*{#1}%
\else
\vspace{.05in}\footnotesize
\parindent .2in
{\upshape\bfseries #1. }\ignorespaces 
\fi}
{\if@twocolumn\else\par\vspace{.1in}\fi}
\newenvironment{AMS}{\begin{@abssec}{\AMSname}}{\end{@abssec}}

\newcommand{\qed}{\nobreak \ifvmode \relax \else
\ifdim\lastskip<1.5em \hskip-\lastskip
\hskip1.5em plus0em minus0.5em \fi \nobreak
\vrule height0.75em width0.5em depth0.25em\fi}
\newcommand{\CB}{\beta}
\newcommand{\UB}{\mathfrak{U}}

\newcommand\keywordsname{Key words}
\newcommand\AMSname{AMS subject classifications}

\begin{document}
\title{Avoiding extremes using Partial Control.}
\author{Suddhasattwa Das\footnotemark[1], 
\and James A Yorke\footnotemark[2]
}

\footnotetext[1]{Department of Mathematics, University of Maryland, College Park}
\footnotetext[2]{Institute for Physical Science and Technology, University of Maryland, College Park}

\date{\today}
\maketitle

\begin{abstract}
Dynamical systems can be prone to severe fluctuations due to the presence of chaotic dynamics. This paper explains for a toy chaotic economic model how such a system can be regulated by the application of relatively weak control to keep the system confined to a bounded region of the phase space, even in the presence of strong external disturbances. Since the control here is weaker than the disturbance, the system cannot be controlled to a particular trajectory, but under certain circumstances it can be partially controlled to avoid extreme values. Partial control depends on the existence of a certain set called a ``safe sets''. We describe the safe set and how it varies with parameters, sometimes continuously and sometimes discontinuously. 
\end{abstract}

\begin{keywords}Partial control, chaotic systems, safe sets\end{keywords}

\begin{AMS} 93Cxx\end{AMS}

\section {Introduction}

Partial control (\cite{PartialControl_1}, \cite{PartialControl_2}, \cite{PredPrey}) concerns a situation in which there is a map $f:\mathbb{R}^d\rightarrow \mathbb{R}^d$ and a compact region $Q \subset \mathbb{R}^d$ in which the dynamics of $x_{n+1} = f(x_n)$ are chaotic and for almost every initial point $x_0$, the trajectory $f^n(x_0)$ eventually leaves $Q$.

A bounded disturbance $\xi_n$ and a bounded feedback control $u_n$ are added to $f$, and the goal of the controller is to keep the trajectory of $x_{n+1} = f(x_n) + \xi_n + u_n$ confined to $Q$. Here $u_n$ is chosen with knowledge of $f(x_n) + \xi_n$.
We view $\xi_n$ as the cumulative result of ongoing disturbances over the time interval $(n, n+1]$, the time since the last control input, and these ongoing disturbances are observed by the controller as they occur so that the controller is ready to respond at the end of that interval with the response $u_n$.
The control goal is easy to achieve if the control bound $\UB$ is larger than the disturbance bound $\CB$ but we investigate situations in which $\UB<\CB$. In our case it is impossible to select an unstable trajectory and choose the control so that that trajectory is followed.
The strategies of choosing $u_n$ depend on the bounds $\UB$ and $\CB$ and this paper investigates how the strategy depends on these bounds in the case where f is a one-dimensional piecewise expanding map.

The subject of controlling chaotic systems has been dealt with in several papers in the past so it is important to clarify how partial control is different. In \cite{Controlling_chaos}, the authors demonstrate that chaotic systems can be controlled by making small time-dependent perturbations to a chaotic system so as to steer the state to a nearby periodic trajectory. Their method is applicable to systems whose dynamics is not known. In contrast, the method of partial control requires an explicit knowledge of the map f and aims at preventing only escape from $Q$ rather than targeting some reference trajectory. 
We begin with a specific discrete time example to illustrate the nature of this general control-problem \cite{PartialControl_1}, \cite{PartialControl_2}. Then we investigate how the problem changes as the parameters change.

In \cite{PredPrey}, a method of partial control was proposed to sustain a 3-species predator-prey system with chaotic dynamics. The amount of control needed to avoid the extinction of a species was demonstrated to be smaller than that needed by classical control methods, even in the case where $\xi_n$ is chosen purposefully to drive the trajectory from $Q$. In \cite{ControlFreq}, the authors demonstrated that the use of partial control leads allows one to apply the control after larger time intervals. They investigated the minimum control frequency needed for partial control in a 1D tent map and the Henon map.

The requirement that the trajectory must stay in a specified region has been discussed in the control literature as \textbf{set invariance}. 
The central question that Bertsekas and others (\cite{Reachability1},\cite{Reachability2},\cite{Reachability3}) investigate is ``under what conditions can the state of the uncertain system be forced to stay in a specified region of the state space for all times by using feedback control'' (\cite{Reachability1}) . 
This is also our question, but in a discrete-time setting. They require the control vector to be of lower dimension than that of the state vector since in their framework it would be trivial to control the trajectory if the dimensions were equal. We however set the dimensions equal and in this paper both are a single variable, that is, one dimensional. The problem is not trivial here because the bound on the control is smaller than the bound on the disturbance.

Our main example uses a one-dimensional tent map and for motivational purposes, we think of it as an economic model. Any one-dimensional model of the growth rate $x$ of an economy and its dynamics is simplistic.
It is likely that any model of any dimension could be called too simplistic -- because world economies are truly complicated. 
However the model is complicated enough to introduce complicated phenomena including dangers of economic crashes as the world has recently seen. 
We have also seen that the world economy appears to be unstable, amplifying the effects of some disturbances, and our model incorporates that sensitivity. 
This paper is not aimed at saving the world economy but rather at introducing into the control literature a control strategy that we believe the control community might find valuable. 
One can argue that the controls available to governments are small compared with the disturbances
and our model addresses that problem of having controls that are weaker than the disturbances. 
The strategy is to respond to the disturbance by driving that trajectory to a maximal ``safe'' set. This set is not invariant, and not even connected, despite the one dimensionality of the problem.

\section{A toy example}
The 1-dimensional map we will use as an example is shown in Fig.(\ref{fig:Toy_economic_map}). It is called an asymmetric tent map and is described below.
\begin{equation}
\label{eqn:tent_map}
f(x)=\begin{cases}
1.3x, & \mbox{for } 0\leq x\leq 0.7\\
0.91-3(x-0.7), & \mbox{for } 0.7\leq x\leq 1
\end{cases}
\end{equation}
Tent maps can be used to model a variety of physical, biological, and engineering applications, but here we restrict attention to an economic growth model. Here $x$ is scaled so that $x=0$ corresponds to the zero growth rate (during a severe depression in the economy) and $x=1$ corresponds to some unreachable growth rate. The growth rate in year $n$ is $x_n$. 

The dynamics, is (without external disturbances and control), 
\begin{equation}\label{eqn:basic_map}
x_{n+1}=f(x_n)
\end{equation}
\begin{figure}
\centering
\includegraphics[scale=0.08]{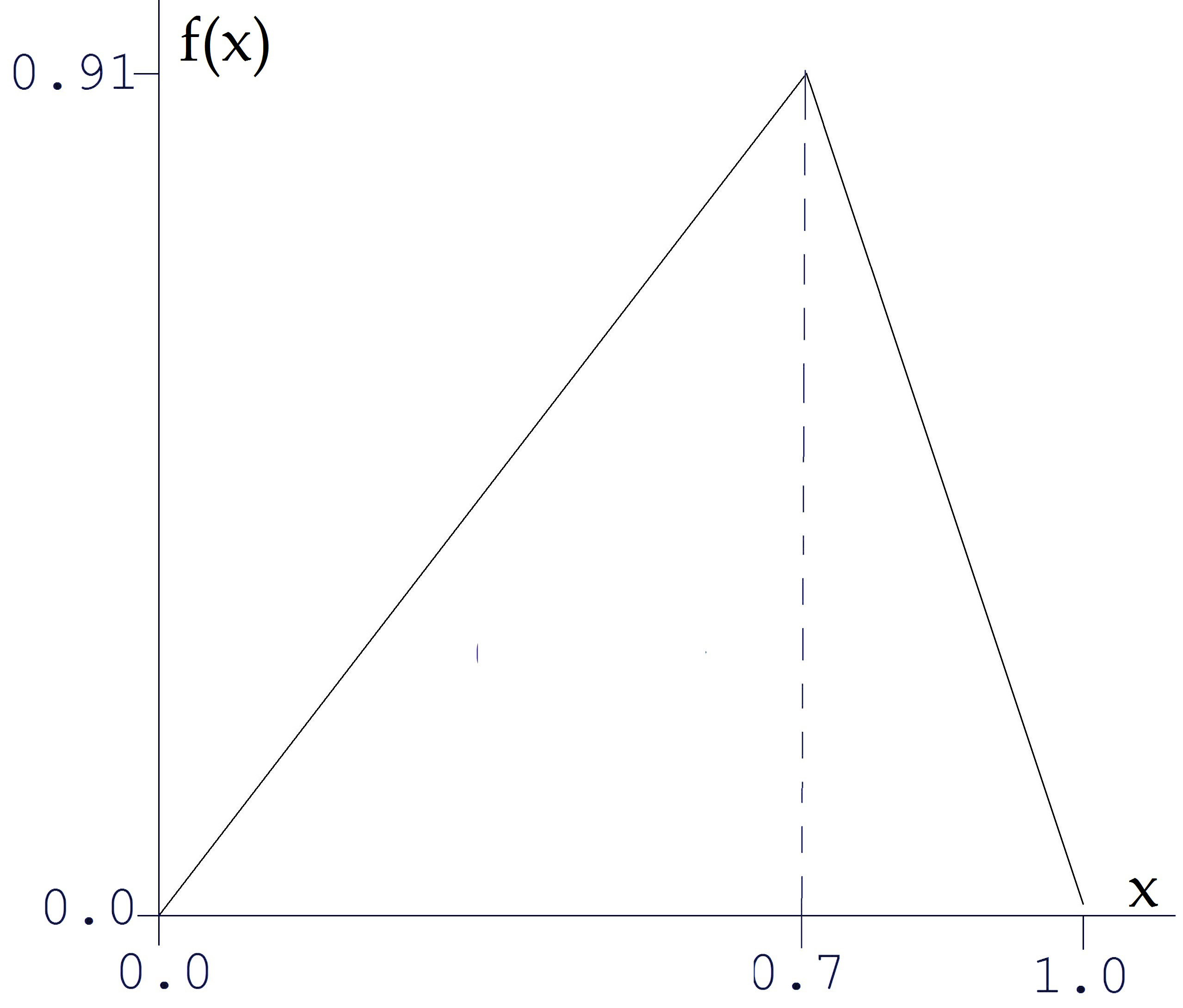}
\caption{\textbf{Asymmetric tent map.} The piecewise linear and expanding map $f$ from Eqn \ref{eqn:tent_map} has slopes $1.3$ and $-3$}
\label{fig:Toy_economic_map}
\end{figure}
Sometimes, allowing an economy to grow too fast can lead to a crash in the economy. That feature is seen in Fig. (\ref{fig:Iterates}), values of $x_n$ near $0.7$ are followed by $x_{n+1}$ near $0.91$, in turn leading to a crash with $x_{n+2}$ near $0.28$. Subsequent recovery is slow. Without any control, the quantity x has repeated crashes. 

\begin{figure}
\centering
\includegraphics[scale=0.14] {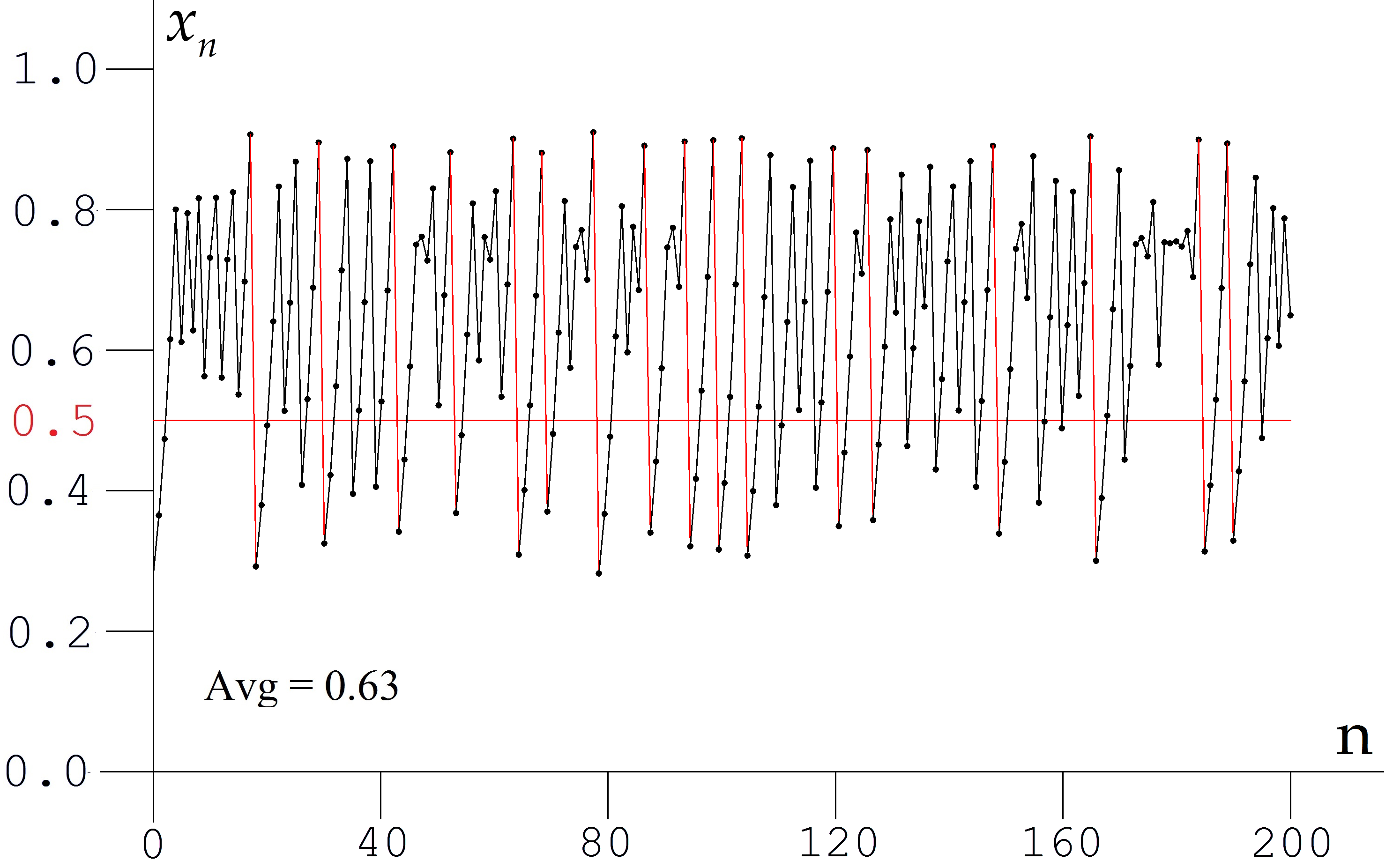}
\caption{\textbf{A chaotic trajectory.} The first 200 iterates of the asymmetric tent map (Eqn. \ref{eqn:tent_map}) are shown. Notice that whenever $x_n$ is near $0.7$, $x_{n+1}$ is below the red line at $0.5$ which denotes our threshold for an economic ``crash''. These crashes have been indicated as red lines. 
}
\label{fig:Iterates}
\end{figure}

\textbf{Perturbed map.} In the presence of strong perturbations, the trajectory can fluctuate even more wildly. We add an external perturbation $\xi_n \in \mathbb{R}$, hence :
\begin{equation}\label{Eqn:perturb}
x_{n+1}=f(x_n)+\xi_n, \mbox{ where } |\xi_n|\leq \CB
\end{equation}
Any perturbation satisfying $|\xi_n|\leq \CB$ will be called an \textbf{admissible perturbation}, where $\CB>0$ is a fixed bound. 

A perturbed trajectory with $\CB=0.05$ plotted in Fig.(\ref{fig:Disturbed_iterates}) shows how disturbances can lead to crashing and prolonged depression in the economy. The average of $x$ dropped from $0.65$ in Fig. (\ref{fig:Iterates}) to $0.59$ in Fig. (\ref{fig:Disturbed_iterates}), because of disturbances. We shall somewhat arbitrarily say that the event of the growth rate falling below $0.5$ will be called a \textbf{\emph{crash}}. Thus we have a target region $Q = [0.5, 1]$ to which we want to confine the trajectory $(x_n)$. It is a compact (bounded, closed) set. The general problem in $\mathbb{R}^n$ of keeping a chaotic trajectory in a closed, bound set $Q$ in the presence of disturbances stronger than control is described in [1], [2]. `

\begin{figure}
\centering
\includegraphics[scale=0.14] {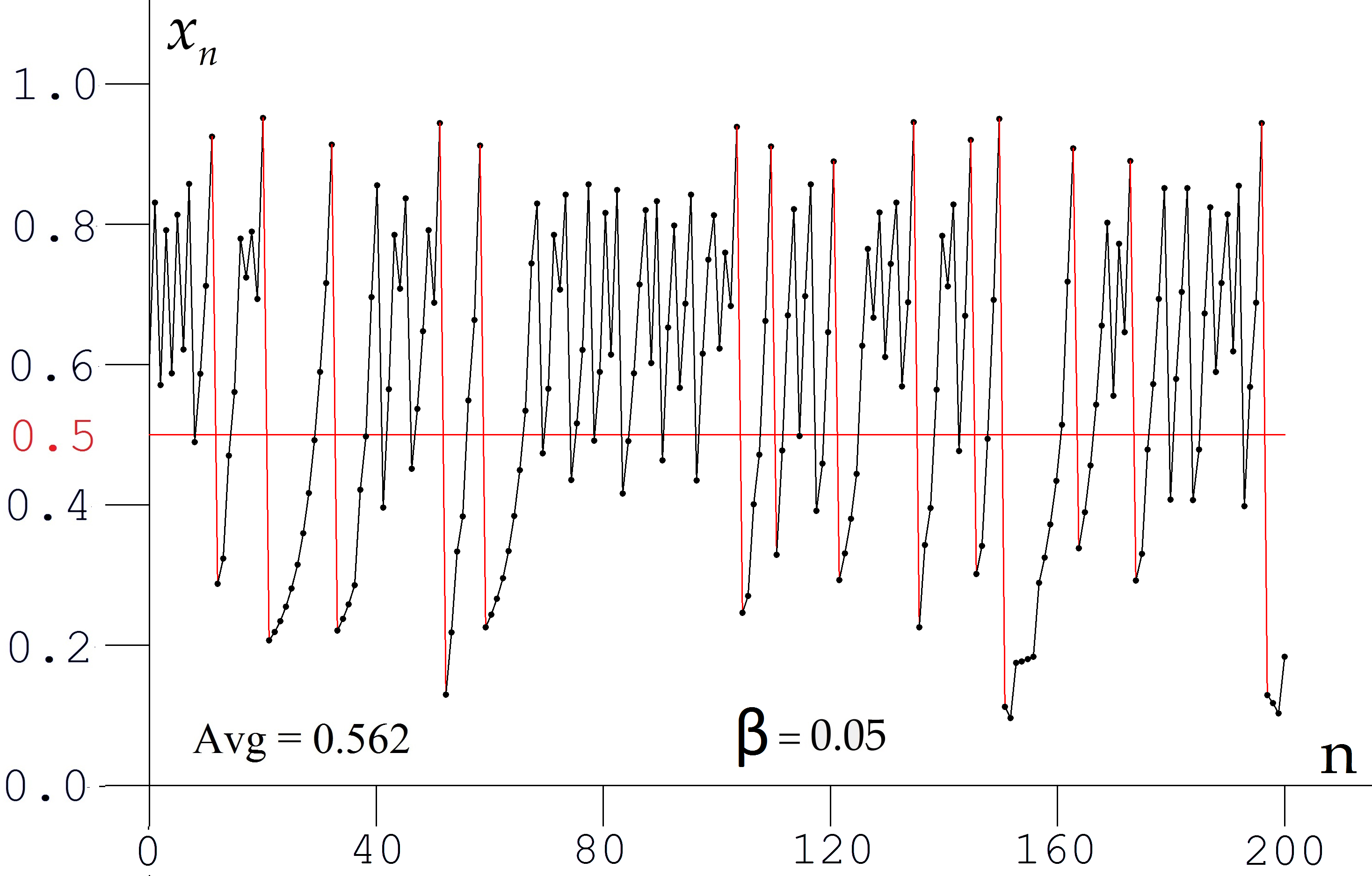}
\caption{\textbf{Perturbed iterates.} The first 200 iterates of the perturbed trajectory (Eqns. \ref{Eqn:perturb}, \ref{eqn:tent_map}) are shown, with the perturbation bound $\CB=0.05$. The average value of $x$ is $0.59$. Each $\xi_n$ is chosen to be $\pm\CB$ so as to exacerbate the crashes.}
\label{fig:Disturbed_iterates}
\end{figure}

\textbf{The perturbed and controlled map.} Knowing the perturbed value $f(x_n)+\xi_n$, we choose a control $u_n$ satisfying $|u_n|\leq \UB$ where $\UB>0$ is a constant. Such a control input is called an \textbf{admissible control}. The primary goal of the control is to choose admissible $u_n$ , given admissible $\xi_n$, so that $x_{n+1}= f(x_n)+\xi_n+u_n \in [0.5,1]$. The resulting dynamics is :
\begin{equation}
\label{eqn:control}
x_{n+1}=f(x_n)+\xi_n+u_n, \mbox{ where }|\xi_n|\leq \CB, \mbox{ and } u_n\leq \UB
\end{equation}
Each $u_n$ is chosen with knowledge of the perturbed value $f(x_n)+\xi_n$, subject to the constraint $|u_n|\leq \UB$ where $\UB>0$ is a constant. Such a control input is called an \textbf{admissible control}. The primary goal is to create a control strategy so that $(x_{n+1})$ always remains in $[0.5,1]$. 

Systems with both state and control constraints have been dealt with in \cite{GutCwi1}, \cite{GutCwi2} and \cite{GutCwi3} in a spirit similar to ours, but for conventional discrete time control systems. The dependence of their control strategy on the state and not the perturbation meant that their methods are not applicable to ours. The application of control \emph{after} a perturbation is also not a new concept. For example, event-based control strategies (see (\cite{EventCtrl}) apply control only when triggered by an event, an event being a subset of the phase space. Event based control strategies are effective in some discrete dynamical systems which are not being monitored in continuous time. In \cite{MinAttent}, the authors use this strategy for designing \emph{minimum attention} control systems.

\begin{figure}
\centering
\includegraphics[scale=0.14] {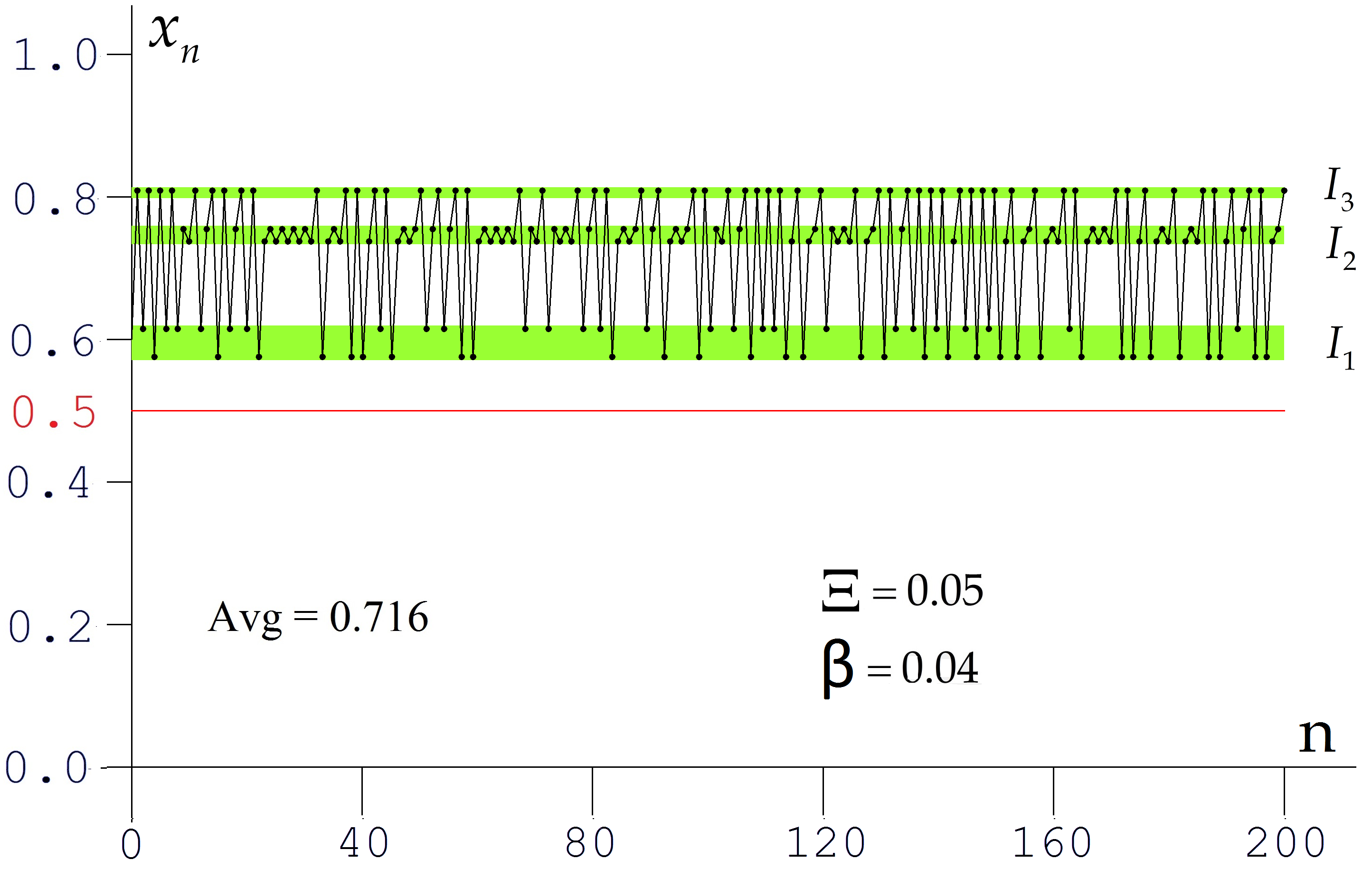}
\caption{\textbf{Partially controlled trajectory.} The first 200 iterates of a partially controlled trajectory (Eqns. \ref{eqn:tent_map}, \ref{eqn:control}) are shown, with the perturbation bound $\CB=0.05$ and the control bound $\UB=0.04<\CB$. The average value of $x$ now is $0.72$. By applying an admissible control bounded by $\UB$ to the same system, the trajectory can be kept above $0.57$ for all times by choosing each $u_n$ so that $x_{n+1}$ is in the green colored set $S=I_1\cup I_2\cup I_3$. Fig. \ref{fig:Safeness_demo} explains why this strategy can be employed to keep $x_n$ in $S$ for all $n$.}
\label{fig:Controlled_iterates}
\end{figure}

\textbf{The crash avoidance strategy.} In Figure \ref{fig:Controlled_iterates}, a strategy for choosing admissible $u_n$ is used that guarantees that the trajectory can be kept in $Q$ for some initial points in $Q$. Simply put, there is a compact set $S$ for which if $x$ is in $S$, then no matter how the admissible $\xi$ is chosen, there is an admissible $u$ (depending on $x$ and $\xi$) such that $f(x) + \xi + u \in S$. In Fig. \ref{fig:Controlled_iterates}, the set $S$ is the union of three intervals (shown in green).

We call such a set $S$ a \textbf{safe set} if it has the following property : for each $x\in S$ and admissible $\xi$, there is a $u$ such that $f(x) + \xi + u \in S$. The choice of $u$ depends on $x$ and $\xi$.

There is a largest safe set, the \textbf{maximal safe set}, and it has the following properties (S1) and (S2), with (S2) actually being stronger than (S1).

(S1) \ If $x\in Q-S$, then there is an admissible $\xi$ such that there is no admissible $u$ for which that $f(x) + \xi + u \in S$.

(S2) \ If $x_0\in Q-S$, then there is a strategy for choosing admissible $\xi_n$ as a function of $x_n$ such that no matter how $u_n$ is chosen, the trajectory can be eventually driven out of $Q$. That is, for some $N$, $x_N \notin Q$.
\\\\
In Figure \ref{fig:Controlled_iterates}, the safe set is the union of the three intervals $I_1\approx [0.57, 0.61]$, $I_2\approx[0.74, 0.75]$ and $I_3\approx[0.80, 0.81]$. $S=I_1\cup I_2\cup I_3$. As mentioned earlier, the strategy is to choose each $u_n$ so that $x_{n+1}$ lies in the safe set. The surprising consequence of being a safe set is that if $x_n\in S$, no matter what admissible $\xi_n$ occur, an admissible $u_n$ can be chosen so that $x_{n+1}= f(x_n)+\xi_n+u_n$ is again in the safe set. This set $S$ is the largest set in $[0.5,1]$ with this property. The sequence of indices $\{i_n\}$ such that $x_n\in I_{i_n}$ is primarily determined by the disturbance $\xi_n$. If $f(x_n)+\xi_n$ is within distance $0.04$ from only one of the three intervals, then $u_n$ has to be chosen so that $x_{n+1}$ is in that same interval.

\begin{figure}
\centering
\includegraphics[scale=0.3] {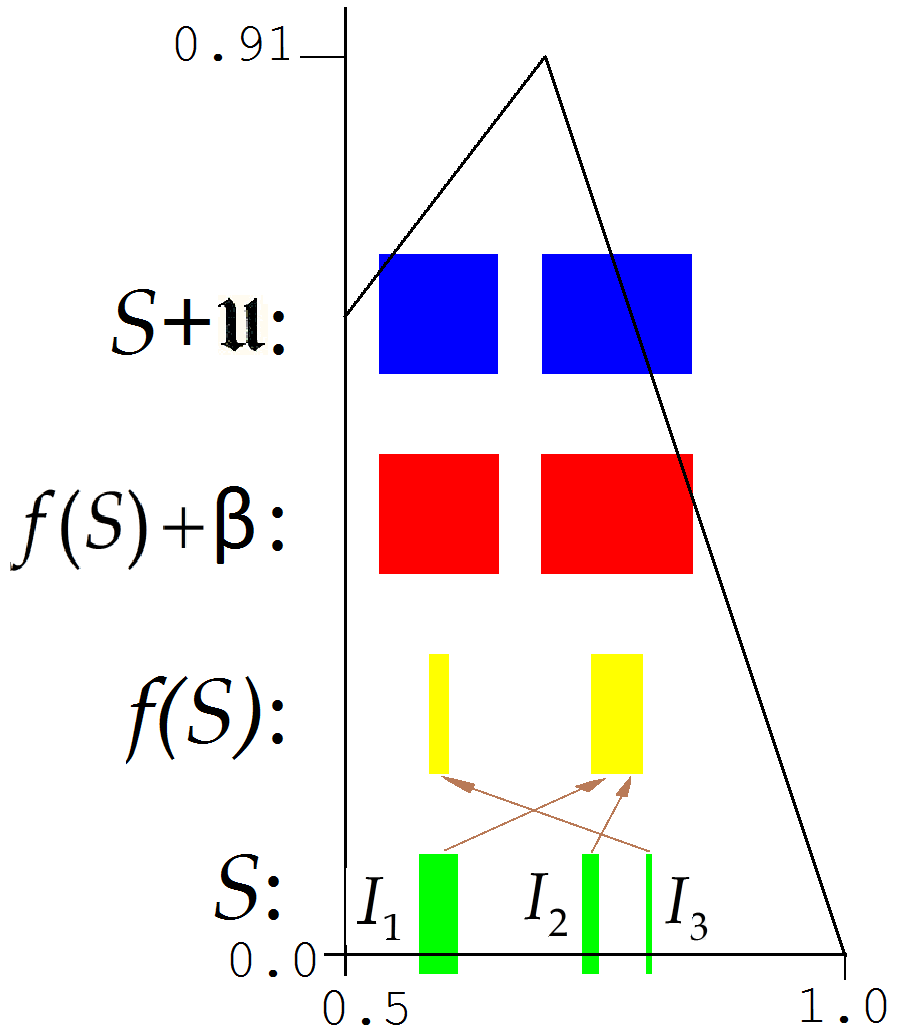}
\caption{\textbf{The safe set and its dynamics.} The colored blocks represent intervals in $[0.5,1]$. The safe set $S$ for Eqns. \ref{eqn:tent_map}, \ref{eqn:control}, where $\UB=0.04$ and $\CB=0.05$ has components $I_1\approx[0.5748, 0.6142]$, $I_2\approx[0.7372, 0.7543]$, $I_3\approx[0.8019, 0.8084]$. Their images under $f$ are $f(I_1)=f(I_2)\approx[0.7472, 0.7984]$, $f(I_3)\approx[0.5848, 0.6042]$, as indicated in the figure by arrows. Let $J_i=f(I_i)+\CB$ denote the interval $f(I_i)$ thickened by $\CB$. Therefore, for all $x_n\in S$ and all admissible $\xi_n$, $f(x_n)+\xi_n$ must be in either $J_1=J_2\approx[0.6972, 0.8484]$ or $J_3\approx[0.5348, 0.6542]$. All the points in this set are within distance $\UB(=0.04)$ of $S$. [Note that $S+\UB$ $\approx$ $[0.5348, 0.6542]\cup[0.6972, 0.8484]$. Hence, there exists admissible $u_n$ so that $f(x_n)+u_n+\xi_n$ is back in the safe
set.}
\label{fig:Safeness_demo}
\end{figure}

\textbf{Unsafe points.} Points not in the maximum safe set are points from where it is possible for admissible perturbations to drive the trajectory below $0.5$ in a finite number of steps, no matter what admissible controls are chosen. 
In Fig. \ref{fig:Iterates}, it is noted in effect that the point $x=0.7$ cannot be in the safe set and hence is an unsafe point. Conversely, the safe set $S$ denotes the set of growth rates $x$ of the economy from which it is always possible to avoid crashes. $S$ depends on the parameters $\UB$ and $\CB$ as well as on the threshold set for crashes, $0.5$. This paper investigates how the safe set changes as $u$ and $\delta$ are changed. In \cite{ContTimeInvThm}, Nagumo provided a necessary and sufficient condition for a set to be a safe-set for continuous time control systems in terms a certain differential property of the boundary of the safe set, called \emph{tangent cone}. In our partial control problem, we also look at geometric properties of the boundary of the safe set in Theorem \ref{thm:bifur_1} and provide necessary and sufficient conditions for the safe set to be a continuous or discontinuous function of the parameters.

\textbf{Maximum safe set.} The safe set may not exist for some values of the bounds $\UB$ and $\CB$. However, if it exists, there is a maximum (largest) safe set in $Q=[0.5,1]$. It is the collection of all safe starting points, and is necessarily a compact set. 

Maximum safe sets represent the subset of the phase space in which the controlled trajectory will ultimately reside. In \cite{PartialControl_2}, the authors defined the \emph{asymptotic safe set} to be that subset of the safe set which is invariant under the control law. In classical control theory, \textbf{controlled invariant / viable set} is defined as the set to which it is possible to return after the application of control. This concept is the analog of safe sets in classical control theory where control is stronger than perturbation. In \cite{SetInvSurvey}, it is discussed how the control law is highly dependent on the geometry of this set. There is an analog of asymptotic safe sets in the field of system engineering, called the \emph{domain of attraction} or \emph{stability domain}, whose theoretic/computational determination is of fundamental importance in these fields (see for example, \cite{AsympStab}). Set-invariance also finds applications in the qualitative analysis of biological systems, as in \cite{InvSetBio} and \cite{PredPrey}. In \cite{Kolmanovski}, the largest invariant set under state and control constraints is used to improve a control-compensator performance.

\section{Safe set as a function of $\UB$ and $\CB$}

Fig(\ref{fig:Safeset_eg}) shows the safe sets at three values of $\UB$ and $\CB$ for Eqns. \ref{eqn:tent_map}, \ref{eqn:control}. The safe set is composed of intervals and depends on the parameters $\UB$ and $\CB$. As seen from the figures, the number of components can vary, as can the measure (sum of sizes of the components).

\begin{figure}
\centering
\begin{subfigure}{0.3\textwidth}
\includegraphics[scale=0.2]{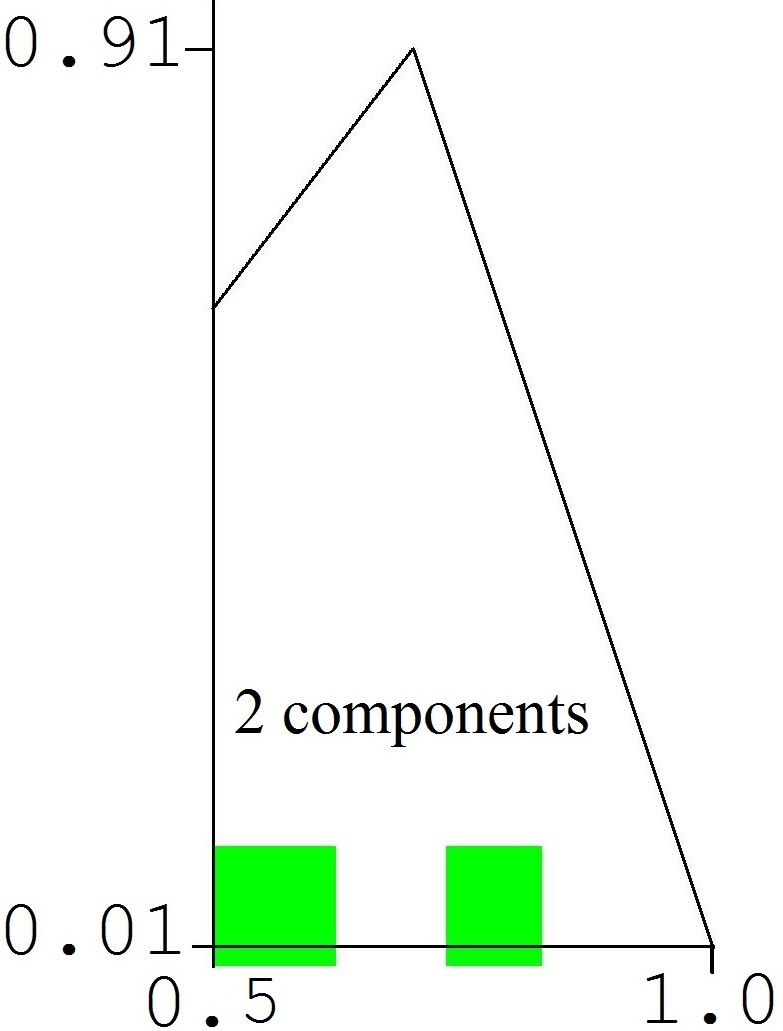}
\caption{$\UB=0.08$, $\CB=0.1$}
\label{fig:Safeset_eg_1}
\end{subfigure}
\begin{subfigure}{0.3\textwidth}
\includegraphics[scale=0.2]{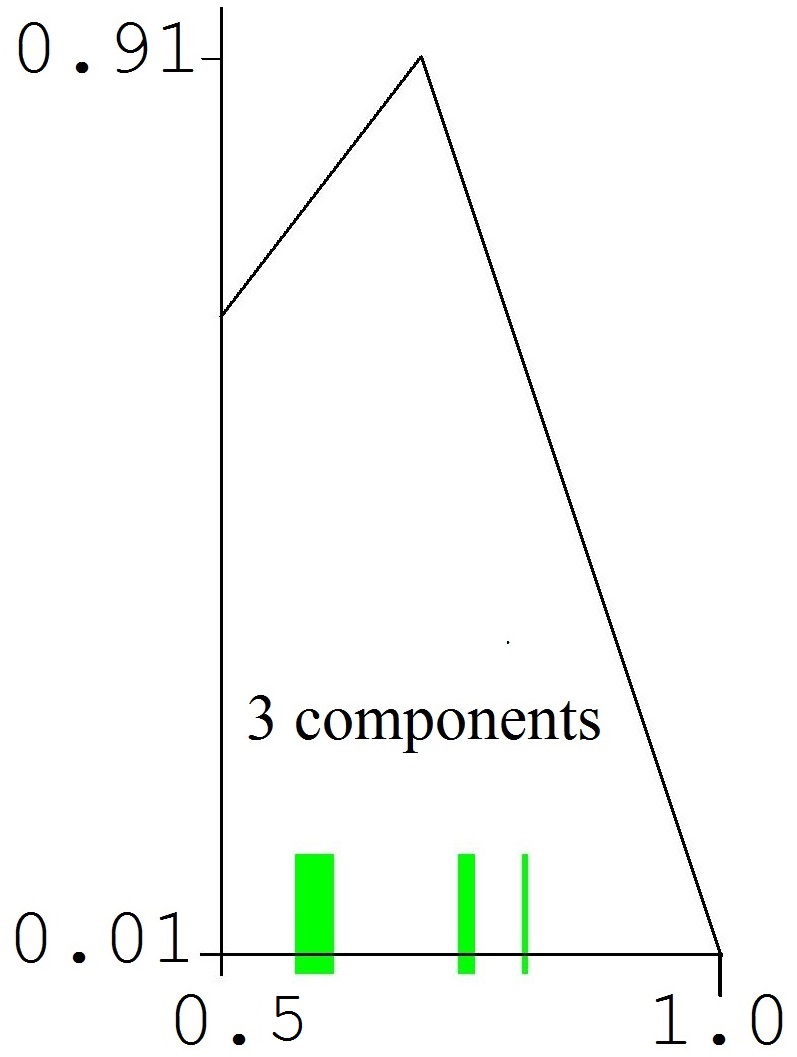}
\caption{$\UB=0.04$, $\CB=0.05$}
\label{fig:Safeset_eg_2}
\end{subfigure}
\begin{subfigure}{0.3\textwidth}
\includegraphics[scale=0.2]{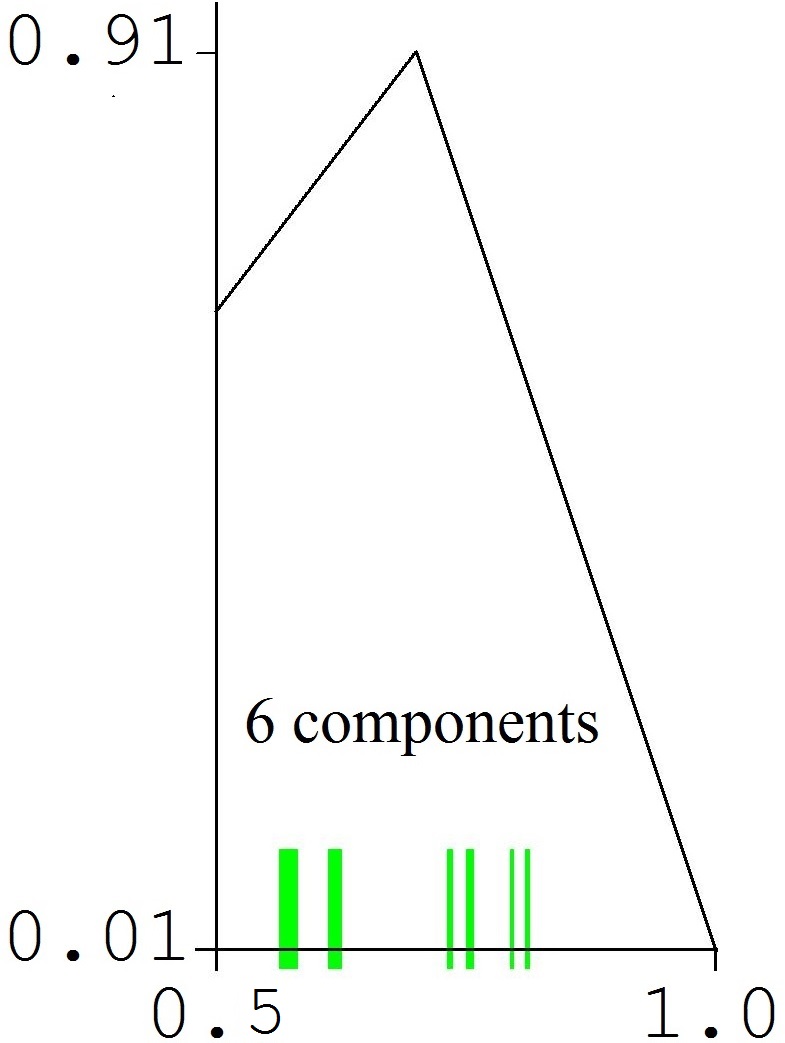}
\caption{$\UB=0.008$, $\CB=0.01$}
\label{fig:Safeset_eg_3}
\end{subfigure}
\caption{\textbf{Examples of safe sets.} The safe sets for the Eqns. \ref{eqn:tent_map}, \ref{eqn:control} are shown, for three choices of $\UB$ and $\CB$. The safe sets $S$ are marked in green, with the number of components of $S$ labeled.}
\label{fig:Safeset_eg}
\end{figure}

Fig. \ref{fig:Measure} shows a plot of the measure of the safe set as functions of $\UB$ and $\CB$, with $0<\UB<\CB$. The number of components has also been marked for some regions. The measure (sum of lengths of components) has been indicated by color.

\begin{figure}
\centering
\includegraphics[scale=0.1] {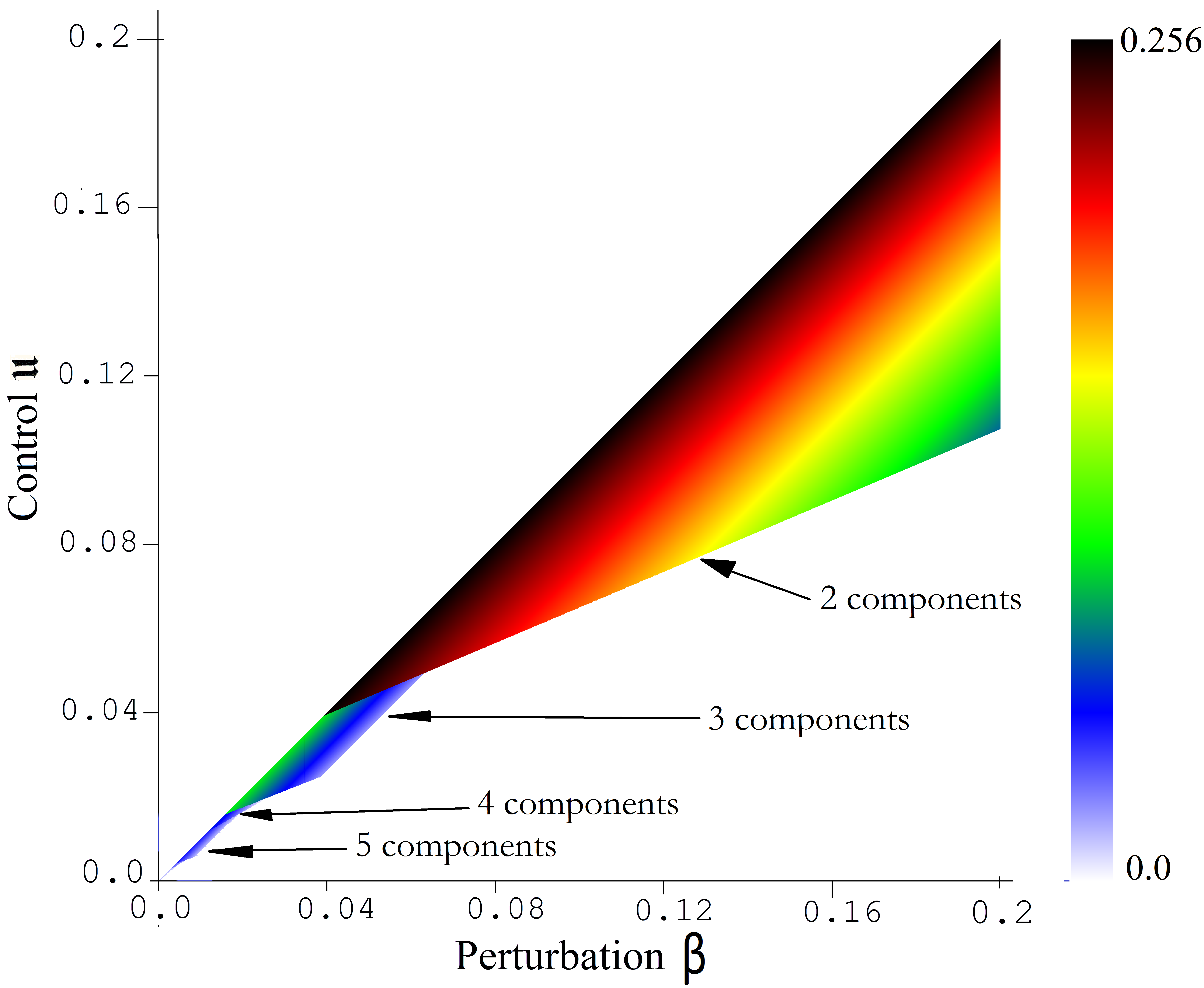}
\caption{\textbf{Size of the safe set} The colored region in this plot denotes the values of $(\UB,\CB)$, ($0 < \UB < \CB ≤ 0.2$), at which there is a safe set for the asymmetric tent map Eqns. \ref {eqn:tent_map}, \ref{eqn:control}. There are no points on this plot above the diagonal $\UB=\CB$ since the control bound $u$ is always less than the disturbance bound $\CB$. The blank space at the bottom of the figure denotes values of $(\UB,\CB)$ for which there is no safe set. The legend on the right shows how the colors correspond to measure (i.e., sum of lengths of components) of the safe set as a function of $\UB$ and $\CB$. The number of components in the safe set has also been marked for certain portions of the plot.}
\label{fig:Measure}
\end{figure}

From the plot, some of the observations that can be made are :
(i) For every value of the perturbation $\CB$, there is a minimum value of the control $\UB$, denoted as $u_{min}(\CB)$ for which $S$ is nonempty.
(ii) Some portions of the graph of $u_{min}$ seem to be straight lines with slope $1$.
(iii) There are certain boundaries in the plot across which the number of components increase on decreasing $\UB$ or increasing $\CB$.

Fig. \ref{fig:n_profile} shows how the maximal safe set varies with $\UB$ for fixed $\CB=0.05$. The safe set has $2$ components for $\UB\geq \UB^*\approx 0.45$. At $\UB=\UB^*$ one of the components splits into two smaller components. At $\UB\approx 0.0357$, one of the components shrinks to a single point and below that the entire safe set vanishes.

\begin{figure}
\centering
\includegraphics[scale=0.15] {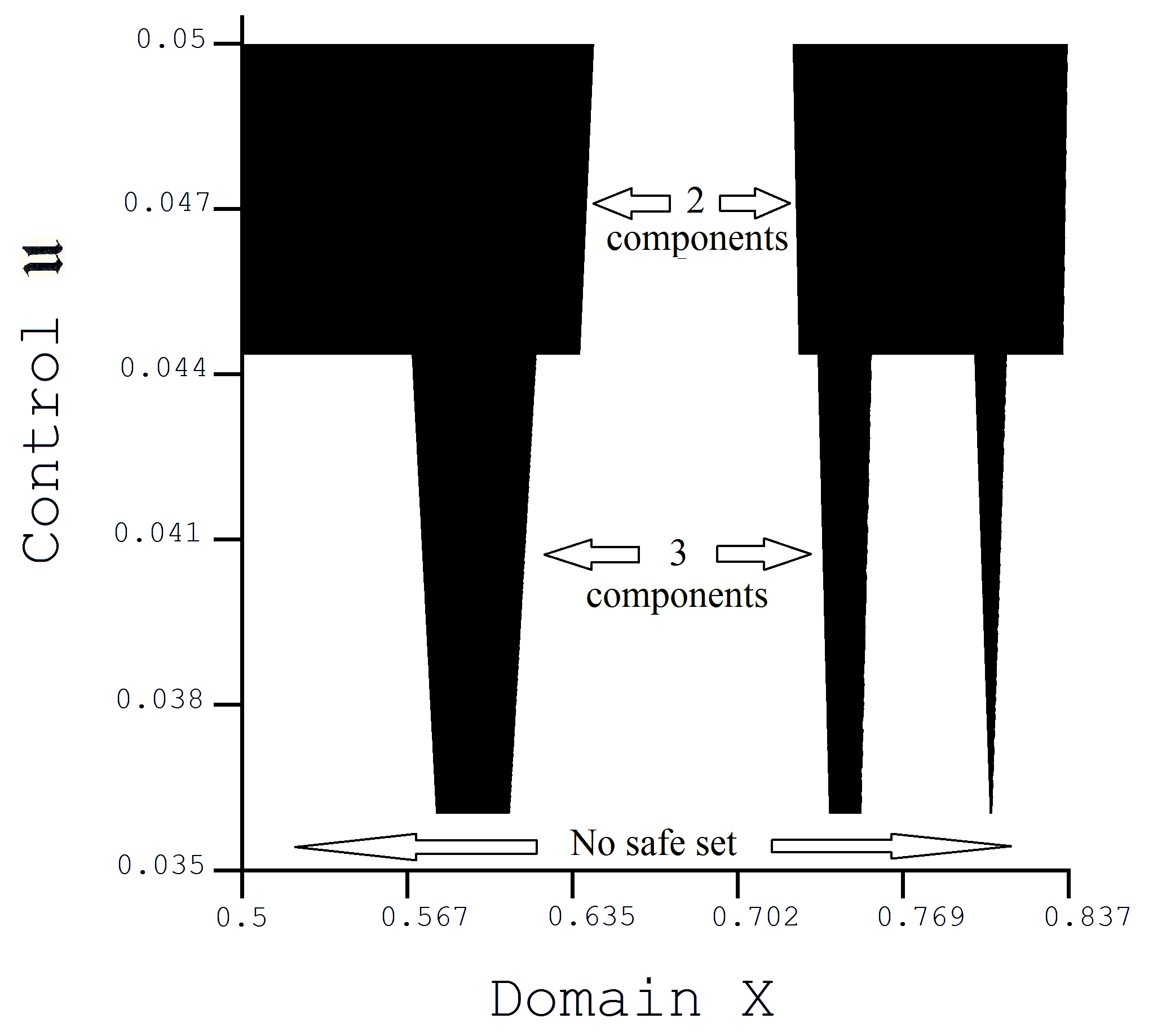}
\caption{\textbf{ The safe set of Eqns. \ref{eqn:tent_map}, \ref{eqn:control}, as a function of \boldmath $\UB$, with $\CB$ fixed at $0.05$ \unboldmath.}}
\label{fig:n_profile}
\end{figure}

(i) At $\UB\approx 0.45$, the distance between the two component intervals is $2\UB$, which results in one of the components splitting discontinuously into two smaller components as $\UB$ is decreased. We call such a bifurcation a \textbf{split bifurcation} in the next section. 
\\(ii) At $\UB\approx 0.0357$, one of the component intervals shrinks to a single point, which in general results in one or more components disappearing as $\UB$ is decreased. Here, the entire safe set actually vanishes. We call such a bifurcation a \textbf{vanishing point bifurcation} in the next section. 

\section{Continuity of Safe sets}
\subsection{The bifurcation theorem}
Throughout this chapter, we investigate the system described in Eqn. \ref{eqn:control} and assume that the target $Q$ is a closed, non-degenerate interval $[A, B]$, where $A < B$. $S_{\UB,\CB}$ will denote the maximum safe set at control bound $\UB$ and perturbation bound $\CB$. Recall that the maximum safe set for the target set $Q$ is the set of states $x$ starting from which a trajectory can always remain in $Q$, no matter what admissible perturbations occur. 

\begin{definition} $S_{\UB,\CB}$ is \textbf{continuous} with respect to $\UB$ and $\CB$ at $(\UB_0, \CB_0)$ if when $(\UB,\CB)$ is varied in some sufficiently small neighborhood $\mathcal{N}$ of $(\UB_0, \CB_0)$, the number of components in $S_{\UB,\CB}$ remains the same and the boundaries change continuously with $(\UB,\CB)$.
\end{definition}

Note that for the asymmetric tent map (Eqn. \ref{eqn:tent_map}), $f$ is not differentiable at $0.7$ and this point is also never part of the safe set.

The following theorem, which lays down sufficient conditions for bifurcations / discontinuities in the safe set to occur, is proved in Section \ref{proof_bifur} after some definitions are introduced.

\begin{theorem}[Bifurcation theorem]\label{thm:bifur_1}
Let the map $f$ (in Eqn. \ref{eqn:control}) be continuous, piecewise expanding and piecewise $C^1$. Let $Q$ be a compact interval. Let $0<\UB<\CB$ be the control and perturbation bounds respectively and assume that there is a nonempty maximum safe set $S_{\UB,\CB}$ and that $f$ is differentiable in a neighborhood of $S_{\UB,\CB}$. Then if a bifurcation of the safe set $S=S_{\UB,\CB}$ occurs at $(\UB,\CB)$ then at least one of these conditions hold :\\
\textbf{(B1)} A component interval of $S$ is a point.\\
\textbf{(B2)} The gap between two adjacent intervals of $S$ equals $2\UB$.\\
Moreover, (B1) is also sufficient for a bifurcation to occur.
\end{theorem}

If a bifurcation occurs at $(\UB,\CB)$, we say that it is a \textbf{vanishing point bifurcation} if (B1) holds and a \textbf{split bifurcation} if (B2) holds but not (B1).

\subsection {Definitions and properties}

As usual, $d(x,y)$ will be used denote the Euclidean distance between 2 points $x$ and $y$ on the real line. For a set $A\subset\mathbb{R}$, $\partial A$ denotes the boundary of $A$ , $A^C$ the complement of $A$ and $\bar{A}$ the closure of $A$.

\begin{definition} Let $X\subset \mathbb{R}$. Then for every $u>0$, let \boldmath $X+u$ \unboldmath denote the closed set of points that are within distance $u$ from $X$. In other words, $X+u$ is the set $\{x\in\mathbb{R}|\ \exists y\in X\ni d(x, y)\leq u\}$. An equivalent definition is $X+u = \underset{x\in X}{\cup}\bar{B}(x,u)$, where $\bar{B}(x,u)$ represents a closed ball of radius $u$ and center $x$.
\end{definition}
Thus if $X$ is compact, then $X+u$ is a compact set.

For a given compact set $Q$, control bound $\UB$ and disturbance bound $\CB$, the following equation from \cite{PartialControl_2} gives an equivalent formulation for a set $S\subseteq Q$ to be a safe set. It is a restatement of the property that if $x\in S$, then $f(x)+\xi$ is within distance $u$ of $S$.
\begin{equation}\label{eqn:safe_set_1}
f(S)+\CB\subseteq S+\UB
\end{equation}
Equation (\ref{eqn:safe_set_1}) leads to a constructive proof of existence of the safe set, as shown in \cite{PartialControl_2}. It was called the \textbf{sculpting algorithm} as it produces the safe set as the limit set of an infinite sequence of diminishing compact sets. In general, iterative methods leading to invariant or optimal sets for nonlinear constraint problems often result in sets which are maximal with respect to these properties but with high geometric complexity. They were first employed in feedback-control systems in \cite{Reachability1}, \cite{Reachability2} and \cite{Reachability3}.

\textbf{Maximum safe set.} Let the parameters $\UB$ and $\CB$ be fixed. Any union of safe sets corresponding to $(\UB,\CB)$ is also a safe set for $(\UB,\CB)$; hence, the union of all safe sets for these parameters is the unique maximal safe set. Since the closure of a safe set is also a safe set, the maximum safe set must be a compact set if $Q$ is compact. 

If $f : Q\rightarrow \mathbb{R}$ is continuous, $Q$ compact and $0 < \UB < \CB$, then a maximal safe set \boldmath $S_{\UB,\CB}$ \unboldmath satisfies the stronger equation (see Appendix \ref{lemma:maximal}) :
\begin{equation}
\label{eqn:maximum_safe}
f(S_{\UB,\CB})+\CB=[f(Q)+\CB]\cap[S_{\UB,\CB}+\UB]
\end{equation}

Henceforth, the term ``safe set'' will be used to denote ``maximum safe set".

The following definition of a safe set has already been introduced :

\begin{definition} $S_{\UB,\CB}$ is \textbf{continuous} with respect to $\UB$ and $\CB$ at $(\UB_0, \CB_0)$ if $(\UB,\CB)$ is varied in a neighbourhood $\mathcal{N}$ of $(\UB_0, \CB_0)$, then the number of components in $S_{\UB,\CB}$ remains the same and the boundaries change continuously with $(\UB,\CB)$.
\end{definition}

\subsection {Proof of the bifurcation theorem}\label{proof_bifur}

The Bifurcation Theorem \ref{thm:bifur_1} gives necessary conditions condition for bifurcations to occur. We will first prove that condition (B1) implies that a bifurcation occurs.

\textbf{Suppose (B1) occurs.} That is, one of the components of the safe set $S_{\UB,\CB}$ is a single point $\{p\}$. Therefore, the closed ball $\bar{B}(f(p),\CB)$ is a subset of $S_{\UB,\CB}+\UB$. In fact, the component interval $I$ of $S+\UB$ which contains $\bar{B}(f(p),\CB)$ shares a boundary point with $\bar{B}(f(p),\CB)$. For otherwise, if $\bar{B}(f(p),\CB)$ is in the interior of $I$, then since $f$ is continuous, by choosing a small interval $J$ around $p$, $f(J)+\CB\subset I$ will be satisfied. Then $S_{\UB,\CB}\cup{J}$ is a larger safe set, contradicting the maximality of $S_{\UB,\CB}$. So $I$ does share a boundary point with $\bar{B}(f(p),\CB)$.

Now keeping $\CB$ fixed, if $\UB$ is decreased, for every $\delta>0$, the components of $S_{\UB,\CB}+(\UB-\delta)$ will shrink. Since $S_{\UB-\delta,\CB}\subseteq S_{\UB,\CB}$, we get that $S_{\UB-\delta,\CB}+(\UB-\delta)$ is strictly in the interior of $S_{\UB,\CB}+\UB$ . Therefore, $\bar{B}(f(p),\CB)$ cannot be a subset of $S_{\UB-\delta,\CB}+(\UB-\delta)$ and hence $p\notin S_{\UB-\delta,\CB}$. Therefore, the component $\{p\}$ vanishes and a bifurcation occurs.

\textbf{(B1), (B2) are necessary.} We will prove that if neither of conditions (B1) and (B2) are satisfied at some $(\UB_0, \CB_0)$, then the safe set $S_{\UB,\CB}$ changes continuously with $(\UB,\CB)$ for $(\UB,\CB)$ near $(\UB_0,\CB_0)$, ie , the number of components in $S_{\UB,\CB}$ remains constant and if $x_1, x_2, \ldots$ are its boundary points, then they can be expressed as continuous functions of $(\UB,\CB)$ : $x_1(\UB,\CB), x_2(\UB,\CB), \ldots$.

We will first track how the safe set changes when $\CB$ is kept fixed at $\CB_0$ and $\UB$ is decreased below $\UB_0$. Since $\CB$ will be kept constant, $S_\UB$ will be used to denote the safe set $S_{\UB,\CB}$. The proof has two parts : first we establish the equations which the boundary points of $S_{\UB_0}$ satisfy and then we invoke the implicit function theorem to prove that the solutions to these equations vary continuously with $\UB$ and hence, so do the boundaries of $S_{\UB}$. 

Since $S_\UB$ is compact, it is a disjoint union of closed, bounded intervals. So condition (B1) does not hold iff all of these intervals are proper, that is, have non-zero length.

For each boundary point $x_i$, consider the closed ball $\bar{B}(f(x_i),\CB_0)$. By Eqn. \ref{eqn:safe_set_1}, at least one of the following two cases must be satisfied :
\begin{enumerate}
\item $\bar{B}(f(x_i),\CB_0)$ lies in the interior of $S_{\UB_0}+\UB_0$. Then $x_i$ must be a point lying on the boundary of $Q$, otherwise, $S_{\UB_0}$ would not have been maximal. Then if we expect $S_\UB$ to change continuously for $\UB$ near $\UB_0$, then $\bar{B}(f(x_i),\CB_0)$ continues to remain in the interior of $S_\UB+\UB$ and hence $x_i$ continues to be an element of $S_u$. Since it is still on the boundary of $Q$, $x_i(\UB)=x_i(\UB_0)$ is always a boundary point of $S_{\UB}$ for $\UB$ near $\UB_0$. In particular, 
\begin{equation} \label{eqn:non_moving}
\mbox{For every } \UB \mbox{ near } \UB_0, f(x_i)=\mbox{constant}
\end{equation}
\item $\bar{B}(f(x_i),\CB_0)$ shares a boundary point with $S_{\UB_0}+\UB_0$. Then $f(x_i)\pm\CB_0=x_{\sigma(i)}\pm \UB_0$, where $x_{\sigma(i)}$ is some boundary point of $S_{\UB_0}$ which is at distance $\UB_0$ from a boundary point of $S_{\UB_0}+\UB_0$. If $S_\UB$ changes continuously for $\UB$ near $\UB_0$, then by the maximality of $S_\UB$, the following equation will hold :
\begin{equation} \label{eqn:moving}
\mbox{ For every } \UB \mbox{ near } \UB_0, f(x_i)\pm\CB_0=x_{\sigma(i)}\pm \UB
\end{equation}
\end{enumerate}

Now $S_\UB$ can have countably many boundary points but since $S_\UB+\UB$ is a compact set with each component having diameter $\geq 2\UB$, $S_\UB+\UB$ has a finite number of components. Therefore, the range of values of $\sigma(i)$ is finite and correspond to those $x_{\sigma(i)}$-s which are at a distance of $\UB$ from the boundary of $S_{\UB}+\UB$. Without loss of generality, let the boundary points be renumbered so that $x_1,\ldots,x_n$ are the $x_{\sigma(i)}$-s.

Therefore, from Eqns. \ref{eqn:non_moving} and \ref{eqn:moving}, it is clear that all the $x_i$-s are determined by the vector $X=(x_1,\ldots,x_n)$ alone. 
Define $F:\mathbb{R}^n\rightarrow\mathbb{R}^n$ as :
\begin{equation} 
F(X)=
\left(\begin{array}{c}
f(x_1)\\
\vdots \\
f(x_n)
\end{array}\right)
\end{equation}
Then all first $n$ equations can be collected together as : 
\begin{equation}
F(X)=MX+\Delta(\UB)
\end{equation}
where : $M$ is an $n$-by-$n$ matrix all of whose entries are $0$ or $1$ and which has at most one non-zero entry in each row; $\Delta(\UB)$ is a column vector whose entries $\in\{0,\pm\UB\pm\CB_0\}$.

This is a $C^1$ system, which by the implicit function theorem has a $C^1$ solution iff its Jacobian with respect to $X$ is invertible at $\UB=\UB_0$. But this Jacobian $=J-M$, where $J$, the Jacobian of $F$ is a diagonal matrix $\emph{diag}(J_1,\ldots,J_n)$, where $J_i=f'(x_i(\UB_0))$. By Appendix \ref{lemma:mat_cyc}, the factors of $det(J-M)$ are either of the form (i) $J_i$ or (ii) $J_{i_1}\ldots J_{i_k}-1$. Since by assumption, $f$ does not have a zero derivative on the safe set, $J_i\neq 0$. A factor of the form (ii) is also $\neq 0$ since $|f'|>1$ everywhere. 

Therefore, there exist continuous solutions in $x_1(\UB),x_2(\UB),\ldots$ of the defining equation Eqn. \ref{eqn:safe_set_1} of a safe set. When $\UB$ is decreased below $\UB_0$, the safe set $S_{\UB}$ is always a subset of $S_{\UB_0}$. Since these solutions are also the unique solutions around $\UB=\UB_0$, $x_1(\UB),x_2(\UB),\ldots$ must the boundary points of the maximum safe set $S_{\UB}$. This means that the number of components of $S_{\UB}$ remain fixed and change continuously. \qed.

Now that we have proved that $S_{\UB,\CB_0}$ changes continuously as $\UB\rightarrow\UB_0^-$, the proof to the Bifurcation theorem will be complete if the following two things can be proven : 
\begin{enumerate}[(i)]
\item $S_{\UB,\CB_0}$ changes continuously as $\UB\rightarrow\UB_0$.
\item $S_{\UB,\CB}$ changes continuously as $(\UB,\CB) \rightarrow (\UB_0,\CB_0)$.
\end{enumerate}

This will be done sequentially through the following two lemmas.

\begin{lemma}[Upper continuity of safe sets]
\label{lem:Upper_cont_safe}
If $S_{\UB,\CB_0}$ changes continuously as $\UB\rightarrow\UB_0^-$, then it also changes continuously as $\UB\rightarrow\UB_0$.
\end{lemma}
\begin{proof}
The lemma will be proved by showing that, in fact, safe sets are always continuous with $\UB$ from above, that is, it always changes continuously as $\UB\rightarrow\UB_0^+$. Since $S_{\UB_0+\delta,\CB_0}$ decreases with decreasing $\delta$, $S_{\UB_0+\delta,\CB_0}$ will decrease continuously to $S_{\UB_0,\CB_0}$ iff :
\begin{equation}
\mbox{For every } 0<\UB<\CB,\ S_{\UB,\CB}=\underset{\delta>0}{\cap}S_{\UB+\delta,\CB}.
\end{equation}
Let $\bar{S}$ denote the set $\underset{\delta>0}{\cap}S_{\UB+\delta,\CB}$.

Since $\CB$ is fixed in the above equation, it will be dropped from the notation. For all $\delta>0$, $S_\UB\subseteq S_{\UB+\delta}$ $\Rightarrow$ $S_{\UB}\subseteq\underset{\delta>0}{\cap}S_{\UB+\delta}$.

Thus, it remains to be proven that $\bar{S}=\underset{\delta>0}{\cap}S_{\UB+\delta}\subseteq S_{\UB}$. It would be sufficient to prove that $\bar{S}$ is a safe set. For all $\delta>0$, $f(S_{\UB+\delta})+\CB\subseteq S_{\UB+\delta}+(\UB+\delta)$ and $\bar{S}\subseteq S_{\UB+\delta}$

$\Rightarrow$ $f(\bar{S})+\CB\subseteq \underset{\delta>0}{\cap}[S_{\UB+\delta}+(\UB+\delta)]=\underset{\delta>0}{\cap}S_{\UB+\delta}+\UB$ by Appendix \ref{lem:u_lemma}.

$\Rightarrow$ $\bar{S}$ is a safe set for parameters $(\UB,\CB)$ $\Rightarrow$ $\bar{S}\subseteq S_{\UB}$, the maximum safe set at parameters $(\UB,\CB)$.\qed
\end{proof}

\begin{lemma}\label{lemma:Hauss_simple_cont}
$S_{\UB,\CB}$ changes continuously as $(\UB,\CB)\rightarrow (\UB_0,\CB_0)$ iff $S_{\UB,\CB_0}$ changes continuously as $\UB\rightarrow\UB_0^-$.
\end{lemma}
\begin{proof}
Consider the coordinates $(\UB,\CB-\UB)$, which are obtained by a smooth, invertible transformation of the coordinates $(\UB,\CB)$. Therefore, it will be equivalent to prove that the safe set $S$ changes continuously with these new coordinates. Continuity with respect to $\UB$ follows from Lemma \ref{lem:Upper_cont_safe}. To prove continuity with respect to the coordinate $\CB-\UB$, we will prove that in fact, if condition (B2) is not satisfied at $(\UB_0,\CB_0)$, then for sufficiently small $\delta$, $S_{\UB+\delta,\CB+\delta}=S_{\UB,\CB}$.

Let us assume that $\delta>0$. The proof for $\delta<0$ will be analogous. To prove the above identity, first note that
\begin{equation}
\mbox{For every }0<\UB<\CB \mbox{ and every } \delta>0,\ S_{\UB,\CB}\subseteq S_{\UB+\delta,\CB+\delta}.
\end{equation}
This follows from Eqn. \ref{eqn:safe_set_1}. So it remains to prove that if (B1) does not hold at $(\UB_0,\CB_0)$, then for sufficiently small $\delta$, $S_{\UB_0+\delta,\CB_0+\delta}\subseteq S_{\UB_0,\CB_0}$, or equivalently, $S_{\UB_0+\delta,\CB_0+\delta}$ is a safe set for the parameters $(\UB_0,\CB_0)$.

Consider any $x\in S_{\UB_0+\delta,\CB_0+\delta}$, which shall be abbreviated as $S(\delta)$ for ease of notation. By definition, the closed ball $\bar{B}(f(x),\CB_0+\delta)\subset S(\delta)+(\UB_0+\delta)$. Then the smaller ball $\bar{B}(f(x),\CB_0)$ must be at a distance of at least $\delta$ from the boundary of $S(\delta)+(\UB_0+\delta)$. Since (B2) is not satisfied at $(\UB_0,\CB_0)$, for every sufficiently small $\delta$, $S(\delta)+\UB_0$ is precisely the set of points in the interior of $S(\delta)+(\UB_0+\delta)$ which are at a distance of at least $\delta$ from the boundary of $S(\delta)+(\UB_0+\delta)$.\\
$\Rightarrow$ $\bar{B}(f(x),\CB_0)\subset S(\delta)+\UB_0$,
$\Rightarrow$ $f(S(\delta))+\CB_0\subseteq S(\delta)+\UB_0$,
$\Rightarrow$ $S(\delta)=S_{\UB_0+\delta,\CB_0+\delta}$ is a safe set corresponding to the parameters $(\UB_0,\CB_0)$. \qed
\end{proof}

This concludes the proof of the Bifurcation theorem. \qed

\subsection {Evolution of the safe set for the given example}

The Bifurcation Theorem \ref{thm:bifur_1} says that for non-bifurcation points, the boundary points of the maximum safe set changes continuously. In this section, we will use the main example of the asymmetric tent map (Eqn. \ref{eqn:tent_map}) to illustrative that the boundary points are the smooth solutions to a system of differential equation. The main theorem is stated here as claim and we proceed to prove it for the specific case under consideration. We will re-visit the case shown in Fig. (\ref{fig:Safeset_eg_2}), where the safe set corresponding to $\UB=0.04$ and $\CB=0.05$ has the three components, $[a_1, b_1]\approx [0.5747, 0.6141]$, $[a_2, b_2]\approx [0.7371, 0.7542]$ and $[a_3, b_3]\approx [0.8019, 0.8084]$. We can make the following claim,

\textbf{Claim :} For $u$ near $0.04$, the boundary points of the safe set change continuously as functions of $\UB$, and satisfy the following :
\[
\left( \begin{array}{c}
\frac{da_1}{d\UB} \\
\frac{db_1}{d\UB} \\
\frac{da_2}{d\UB} \\
\frac{db_3}{d\UB} 
\end{array}\right)
=\frac{1}{f'(a_1)f'(a_2)f'(b_3)-1}
\left( \begin{array}{c}
f'(b_3)-f'(a_2)f'(b_3)-f'(a_1)f'(a_2) \\
f'(a_2)(f'(b_1)^{-1}(f'(a_1)f'(b_3)-1)-f'(a_1)) \\
f'(a_1)f'(b_3)-f'(a_1)f'(a_2)-1 \\
1-f'(a_1)f'(a_2)-f'(a_2)
\end{array}\right). 
\]
and $f(b_2)=a_2+\CB-\UB$, $f(a_3)=b_1+\UB-\CB$.

\textbf{Proof :} Neither of the two bifurcation conditions (i) and (ii) of the Bifurcation Theorem 3.1 are satisfied, since\\
B1 : the two gaps between adjacent intervals are $a_2-b_1\approx 0.123$, $a_3-b_2\approx 0.0477$ while $2\UB=0.08$.\\
B2 : none of the component intervals is a single point.

Denote the safe set at $\UB$ by $S_\UB$, with $\CB$ fixed at $0.04$. According to Def. 3.3, the defining equation for a safe set is $f(S_\UB)=f(Q)\cap (S_\UB+\UB-\CB)$. For $\UB$ near $0.04$, this relation simplifies to :

$f(S_\UB)=S_\UB+\UB-\CB$, because at $\UB=0.04$, $S_\UB+\UB-\CB=[a_1+\CB-\UB, b_1-\CB+\UB]\cup[a_2+\CB-\UB, b_3-\CB+\UB]\approx[0.5847, 0.6043]\cup[0.7471, 0.7985]$ and $f(Q)\approx[0, 0.91] \supset S_\UB+\UB-\CB$. These relations hold for all values of $\UB$ sufficiently near $0.04$, because $f(S_\UB)=S_\UB+\UB-\CB$ will be satisfied here and because of condition (i), $S_\UB+\UB-\CB$ will have three intervals which can be explicitly written down as above as continuous functions of $\UB$.

Moreover we have :
\begin{equation}\label{eqn:sample_proof_system}
\left(\begin{array}{c}f(a_1)\\f(b_1)\\f(a_2)\\f(b_3)\end{array}\right)
=
\left(\begin{array}{c}a_2+\CB-\UB \\b_3-\CB+\UB \\a_3-\CB+\UB \\a_1+\CB-\UB \end{array}\right)
\end{equation}
and
\begin{equation}
f(b_2)=a_2+\CB-\UB, f(a_3)=b_1+\UB-\CB
\end{equation}

Note that in the system of equations given above, the variables $b_2$ and $a_3$ are completely determined by the variables $(a_1, b_1, a_2, b_3)$, hence it is sufficient to consider he solution in these $4$ variables. 

By the implicit function theorem, these $4$ variables, which have been represented as solutions to a set of equation, vary continuously with $\UB$ iff the Jacobian of the system of equations (\ref{eqn:sample_proof_system}) with respect to these $4$ variables is invertible, or equivalently, the equations can be differentiated with respect to $\UB$ to obtain an ordinary differential equation (ODE). The derivative of the above system of equations with respect to $u$ gives 
\[
\left( \begin{array}{cccc}
f'(a_1) & 0 & 0 & 0 \\
0 & f'(b_1) & 0 & 0 \\
0 & 0 & f'(a_2) & 0 \\
0 & 0 & 0 & f'(b_3)
\end{array} \right)
\left( \begin{array}{c}
\frac{da_1}{d\UB} \\
\frac{db_1}{d\UB} \\
\frac{da_2}{d\UB} \\
\frac{db_3}{d\UB} 
\end{array}\right) 
=
\left( \begin{array}{cccc}
0 & 0 & 1 & 0 \\
0 & 0 & 0 & 1 \\
0 & 0 & 0 & 1 \\
1 & 0 & 0 & 0 
\end{array} \right)
\left( \begin{array}{c}
\frac{da_1}{d\UB} \\
\frac{db_1}{d\UB} \\
\frac{da_2}{d\UB} \\
\frac{db_3}{d\UB} 
\end{array}\right)
+
\left( \begin{array}{c}
-1 \\
1 \\
1\\
-1 \\
\end{array} \right).
\]
Rearranging :
\[
\left( \begin{array}{cccc}
f'(a_1) & 0 & -1 & 0 \\
0 & f'(b_1) & 0 & -1 \\
0 & 0 & f'(a_2) & -1 \\
-1 & 0 & 0 & f'(b_3) \\
\end{array} \right)
\left( \begin{array}{c}
\frac{d(a_1)}{d\UB} \\
\frac{d(b_1)}{d\UB} \\
\frac{d(a_2)}{d\UB} \\
\frac{d(b_3)}{d\UB} \\
\end{array} \right)
=
\left( \begin{array}{c}
-1 \\
1 \\
1\\
-1 \\
\end{array} \right).
\]
The matrix on the left has determinant $f'(b_1)[f'(a_1)f'(a_2)f'(b_3)-1]$, which is non-zero since $|f'|>1$ everywhere except 0.7, which is not in the safe set. As a result it can be inverted to obtain an ordinary differential equation in $(a_1, b_1, a_2, b_3)$, which the reader can verify to be the one in the claim. Therefore, under the given conditions, the safe set not only changes continuously but its boundary points also satisfy an ODE.

\subsection {A closer look at the bifurcations}

\textbf{Split bifurcation.} According to the Bifurcation theorem (\ref{thm:bifur_1}), this situation occurs when the distance between two adjacent components of the safe set is equal to $2u$. As a result, either the safe set vanishes or a component splits into two or more components. The following lemma gives a set theoretic version of the split bifurcation condition :

\begin{lemma} 
\label{lemma:split_cond}
Let $X\subset\mathbb{R}$ be compact. Then no two adjacent connected components of $X$ are at distance $2\UB$ from each other iff for all sufficiently small $\delta>0$, $X+\UB-\delta=X+(\UB-\delta)$.
\end{lemma}

Only a split bifurcation can lead to a splitting into 2 or more new components. In Figure 7, which shows the measure of the safe set as a contour plot, two kinds of boundaries can be seen between the different regions in the plot. It is at these boundaries that the bifurcations occur. The values of $(\UB, \CB)$ where split bifurcations occur, must be those boundaries separating regions of the plot with different components. Though they are straight lines for the tent map, is systems like the logistic map, they are generic curves. Figure 8 showed such a split occurring at the point $(\CB=0.05; \UB=0.045)$, which is very near such a boundary.

\textbf{Vanishing point bifurcation.} In Figure 7, the vanishing point bifurcations occur are seen to be occurring along straight lines with slope 1. Theorem \ref{thm:shrink_slope_1} below proves that this indeed is the case. For any perturbation bound $\CB$, there is a value $u_{min}(\CB)$ , which is the minimum of all values of $\UB$ for which the safe set exists. The following theorem states that the set of pairs of the form ($(u_{min}(\CB),\CB)$ form straight lines with slope 1 almost everywhere, wherever the pair is not a split bifurcation point.

\begin{theorem}
\label{thm:shrink_slope_1}
If for some $\CB_0>0$, $(u_{min}(\CB_0),\CB_0)$ is not a split bifurcation point, then the graph of $u_{min}$ is a straight line with slope 1 near $\CB_0$.
\end{theorem}

\begin{proof} Let for small $\delta$, $S_\delta$ denote the maximum safe set at $\UB:=u_{min}(\CB_0+\delta)$, $\CB=\CB_0+\delta$ and $\UB_\delta$ denote the set $u_{min}(\CB_0+\delta)$. Therefore, $S_\delta$ is the maximum safe set at $(\UB_\delta,\CB_0+\delta)$.

Then by assumption, $f(S_0)+\CB_0=S_0+\UB_0$. Taking closed $\delta>0$ balls around the quantities on both sides of this equation, we get
\\$[f(S_0)+\CB_0]+\delta=[S_0+\UB_0]+\delta$ $\Rightarrow$ $f(S_0)+(\CB_0+\delta)=S_0+(\UB_0+\delta)$.
\\$\Rightarrow$ $S_0$ is also a safe set for $(\UB_0+\delta,\CB_0+\delta)$. Therefore, for 
\begin{equation}\label{eqn:shrink_slope_1}
\mbox{For } \forall\delta>0, u_{\min}(\CB_0+\delta)\leq u_{min}(\CB_0)+\delta.
\end{equation}

Let $\Delta_1$ be the maximum distance between adjacent components of $S_0$ which are less than or equal to $2\UB_0$ distance apart. Then since a split bifurcation does not occur at $\UB_0,\CB_0$, $\Delta_1<2\UB_0$. Then $\Delta:=2\UB_0-\Delta_1>0$. So if $\exists y\in\mathbb{R}$ such that $d(y,S_0)>0.5\Delta_1=\UB_0-0.5\Delta$, then $y\notin S_0+\UB_0$

\textbf{Claim.} For $\forall \delta>0$ such that $\delta<0.5\Delta$, $S_0$ is a safe set for $(\UB_0-\delta,\CB_0-\delta)$.

\textbf{Proof.} Suppose the claim is false. Then $\exists x\in S_0$ such that $\bar{B}(f(x),\CB_0-\delta)$ is not a subset of $S_0+(\UB_0-\delta)$. Then $\exists y=f(x)\pm(\CB_0-\delta)\notin S_0+(\UB_0-\delta)$. This implies that $d(y,S_0)>\UB_0-\delta\geq\UB_0-0.5\Delta$. We have seen that this implies that $y\notin S_0+\UB_0$.

But $y\in f(S_0)+(\CB_0-\delta)\subset f(S_0)+\CB_0= S_0+\UB_0$, a contradiction. Hence the claim must be true.
Therefore, it follows from this claim that,
\begin{equation}\label{eqn:shrink_slope_2}
\mbox{For } \forall 0<\delta<0.5\Delta, u_{\min}(\CB_0-\delta)\leq u_{min}(\CB_0)-\delta.
\end{equation}
The constant $\Delta$ can be chosen to be constant for all points in any small neighborhood of $\CB_0$. Therefore, the two inequalities (\ref{eqn:shrink_slope_1}) and (\ref{eqn:shrink_slope_2}) together prove the claim of this theorem.
\qed
\end{proof}

\appendix
\section{A result on matrices}

\begin{lemma}
\label{lemma:mat_cyc}
Let $D = diag(d_1,\ldots,d_n)$ be a diagonal matrix and $M$ an $n\times n$ $0 - 1$ matrix with at most one $1$ in each row. Then the factors of $det(D - M)$ are of the following form :\\
(i) $d_i$ \\
(ii)$d_{i_1}\ldots d_{i_k}-1$.\\
Moreover, a factor of the form $d_{i_1}\ldots d_{i_k}-1$ occurs iff the principal matrix of $M$ with indices $i_1,\ldots,i_k$ is a permutation matrix corresponding to the cyclic permutation $(i_1,\ldots i_k)$.
\end{lemma}
\begin{proof} The proof will be by induction on $n$. The base case $n = 2$ can be verified by enumerating the few possibilities for the matrix $M$.

For general $n$, there can be two cases :\\
(i) there exists a column in $M$ with all entries 0. Without loss of generality, this column is the first column. Then $det(D - M) = d_1 det(D'-M')$, where $D'$, $M'$ are obtained from $D$, $M$ respectively by deleting the first rows and columns. By the inductive assumption, $det(D'-M')$ has the prescribed format and hence so does $det(D - M)$. \\
(ii) All columns in $M$ have a non-zero entry. Since there is at most one $1$ in each row of $M$, there are at most $n$ $1$-s in $M$. Hence $M$ is a permutation matrix. Using the cycle decomposition of permutations, the rows and columns may be permuted (without changing the determinant) so that $D-M$ is in block diagonal form. Hence, its determinant is the product of the determinant of the blocks. Hence, it is sufficient to prove the theorem for the case that $M$ is a cyclic permutation matrix. If $M$ is a cyclic permutation matrix, then $det(D-M) = det(D) - 1$ Hence in either case, $det(D - M)$ has the prescribed format.\qed
\end{proof}

\section{Safe sets}

\begin{proposition}[The maximality criterion]
\label{lemma:maximal} 
The maximal safe set $S_{\UB,\CB}$ satisfies $f(S_{\UB,\CB})+\CB=[f(Q)+\CB]\cap[S_{\UB,\CB}+\UB]$.
\end{proposition}
\begin{proof} For the rest of the proof, let $S$ denote the maximum safe set $S_{\UB,\CB}$. Note that since $S \subseteq Q$, $f(S)+\CB \subseteq [f(Q)+\CB]$. By the definition of a safe set, $f(S)+\CB\subseteq [S+\UB]$. Therefore, $f(S)+\CB$ must be a subset of $[f(Q)+\CB]\cap[S+\UB]$.


Suppose equality does not hold. Then in particular, $f(S)+\CB$ must be a strict subset of $[f(Q)+\CB]$. Since these are both compact sets, there must be a boundary $z$ of $f(S)+\CB$ and in the interior of $f(Q)+\CB$. So there is a point $y$ on the boundary of $f(S)$ at distance $\CB$ from $z$. Note that $y$ must lie in the interior of $f(Q)$, for if $y$ was on the boundary of $f(Q)$, then $z$ would have been on the boundary of $f(Q)+\CB$. 

Now, $y=f(x)$ for some $x\in S$. We will prove that $x$ must be a boundary point of $S$. Take any open neighborhood $U$ of $y$ in $f(Q)$. Since $f$ is continuous, $f^{-1}(U)$ must be an open neighborhood of $x$. So if $x$ was an interior point of $S$, then by choosing $U$ small enough, $f^{-1}(U)$ could be contained inside $S$. This leads to a contradiction, because $U=f(f^{-1}(U))$ and $U$ is not a subset of the image $f(S)$.

Thus $x$ is a boundary point of $S$ and in the interior of $Q$. Since $f$ has been assumed to be piecewise expanding, (4b) if $x$ is perturbed slightly so as to increase $S$, the image $y=f(x)$ would also get perturbed slightly. Hence, $z$, the corresponding point on the boundary of $f(S)+\CB$ would also get perturbed slightly and still lie in the interior of $f(Q)+\CB$. This contradicts the maximality of $S$.\qed
\end{proof}

\begin{lemma} 
\label{lem:continuity_intersection}
Let for $\delta>0$, $K_{\delta}$ be a decreasing sequence of compact sets satisfying $K_{\delta}\subseteq K_{\delta'}$ if $\delta<\delta'$. Let $K=\underset{\delta>0}{\cap}K_{\delta} \neq\Phi$. Then for all $\epsilon>0$, there exists $\delta>0$ such that $d_{Hauss}(K,K_{\delta})<\epsilon$.
\end{lemma}
\begin{proof}
If the contrary is true, then for all $\delta>0$, there exists $x_\delta\in K_\delta$ so that $d(K,x_{\delta})\geq\epsilon$. These $x_\delta$-s have a limit point $\bar{x}$, which satisfies $d(\bar{x},K)\geq\epsilon$. Since the $K_\delta$-s form a decreasing sequence, $\bar{x}\in K_{\delta}$ for every $\delta>0$. Therefore, $\bar{x}\in\underset{\delta>0}{\cap}K_{\delta}=K$, which contradicts the fact that $d(\bar{x},K)\geq\epsilon$.\qed
\end{proof}

\begin{lemma} 
\label{lem:u_lemma}
Let for $\delta>0$, $K_{\delta}$ be a decreasing sequence of compact sets satisfying $K_{\delta}\subseteq K_{\delta'}$ if $\delta<\delta'$. Let $K=\underset{\delta>0}{\cap}K_{\delta} \neq\Phi$. Then $\underset{\delta>0}{\cap}[K_{\delta}+(u+\delta)]=K+u$.
\end{lemma}
\begin{proof} The intersection $\underset{\delta>0}{\cap}[K_{\delta}+(u+\delta)]$ contains the set $K+u$. Let $y\in \underset{\delta>0}{\cap}[K_{\delta}+(u+\delta)]-[K+u]$. Then $d(K,y)=u+\epsilon$ for some $\epsilon>0$. By Lemma \ref{lem:continuity_intersection}, for $\delta$ sufficiently small, $d_{Hauss}(K_{\delta},K)<0.5\epsilon$. So for $\delta<0.5\epsilon$, $d_{Hauss}(K_{\delta}+(u+\delta),K+u)<\epsilon$. But we have assumed that $y\in\underset{\delta>0}{\cap}[K_{\delta}+(u+\delta)]$, so $d(K+u,y)<\epsilon$ and $d(K,y)<u+\epsilon$, a contradiction.\qed
\end{proof}

\bibliographystyle{unsrt}
\bibliography{Bibliography-Partial_control}
\end{document}